\documentclass[11pt]{amsart}
\usepackage{a4wide}

\usepackage{ mathrsfs }
\usepackage{ amsfonts }
\usepackage{ dsfont }
\usepackage{ amsthm }
\usepackage{ amsmath }
\usepackage{ amssymb }
\usepackage{ verbatim }
\usepackage{ tikz-cd }
\usepackage{ mathtools }
\usepackage{ color }
\usepackage[colorlinks, citecolor=blue,backref]{hyperref}
\usepackage{url,amssymb,enumerate,colonequals}

\usepackage[symbol]{footmisc}

\newtheorem{lem}{Lemma}[section]
\newtheorem{thm}{Theorem}[section]

 \newcommand{\ff}{\mathfrak{f}}

\newcommand{\F}{\mathbb{F}}

\newcommand{\Q}{\mathbb{Q}}

\newcommand{\Z}{\mathbb{Z}}
 
\newcommand{\cP}{\mathcal{P}}
\newcommand{\cQ}{\mathcal{Q}}

\newcommand{\cN}{\mathcal{N}}

\newcommand{\fN}{\mathfrak{N}}

\newcommand{\fb}{\mathfrak{b}}  
\newcommand{\fq}{\mathfrak{q}} 
\newcommand{\fp}{\mathfrak{p}} 
\newcommand{\fP}{\mathfrak{P}} 

\newcommand{\OO}{\mathcal{O}}

\newcommand{\Gisom}{\overline{\rho}_{E,p}\sim\overline{\rho}_{\mathfrak{f},\varpi}}
\newcommand{\modpg}{\bar{\rho}_{E,p}}

\DeclareMathOperator{\Cl}{Cl}
\DeclareMathOperator{\Pic}{Pic}
\DeclareMathOperator{\ord}{ord}
\DeclareMathOperator{\Gal}{Gal}
\DeclareMathOperator{\Norm}{Norm}

\title{Fermat's Last Theorem over $\Q(\sqrt{2},\sqrt{3})$}

\author{Maleeha Khawaja}
\address{School of Mathematics and Statistics, Hicks Building, University of Sheffield, Sheffield S3 7RH, United Kingdom}
\email{mkhawaja2@sheffield.ac.uk}

\author{Frazer Jarvis}
\address{School of Mathematics and Statistics, Hicks Building, University of Sheffield, Sheffield S3 7RH, United Kingdom}
\email{a.f.jarvis@sheffield.ac.uk}


\begin{document}

\begin{abstract}
In this paper, we begin the study of the Fermat equation $x^n+y^n=z^n$ over real biquadratic fields.
In particular, we prove that there are no non-trivial solutions to the Fermat equation over $\Q(\sqrt{2},\sqrt{3})$ for $n\geq 4$.
\end{abstract}

\maketitle
\begin{center}
\emph{Dedicated to Iffat (Zaman) Khawaja}\\
\emph{ January 1936 -- January 2022}\\
\end{center}

\section{Introduction}
Since the groundbreaking work of Wiles \cite{Wiles} on the resolution of the Fermat equation over $\Q$, the Fermat equation has been studied extensively over various number fields. Let $K$ be a number field and let $n\geq 3$ be an integer. The Fermat equation over $K$ with exponent $n$ is the equation
\begin{equation}
    \label{Fermat}
    x^n+y^n=z^n, \qquad x,y,z\in K.
\end{equation}
We say a solution $(a,b,c)$ to \eqref{Fermat} over $K$ is trivial if $abc=0$ and non-trivial otherwise.\\

Wiles' method of resolving \eqref{Fermat} over $\Q$ became known as the modular approach.
Thereafter, Jarvis and Meekin \cite{JarvisMeekin} extended this method to prove that
there are no non-trivial solutions to \eqref{Fermat} over $\Q(\sqrt{2})$ for $n\geq 4$.
This was followed by work of Freitas and Siksek \cite{FLTbig, FLTsmall} 
who established a framework on how to resolve \eqref{Fermat} (and more general Diophantine equations) over totally real number fields. 
Furthermore, Freitas and Siksek \cite{FLTsmall} proved 
that there are no non-trivial solutions to \eqref{Fermat} over $\Q(\sqrt{d})$ for $n\geq 4$, where $3\leq d\leq 23$, $d\neq 5, 17$ is a square free integer.
When approaching real quadratic fields with a larger discriminant, they encountered the obstacle of demonstrating the irreducibility of certain Galois representations and eliminating the number of Hilbert newforms that arose as a result of level-lowering.
Michaud-Jacobs \cite{michaudjacobs2021fermats} worked around these obstacles by studying quadratic points on certain modular curves and working directly with Hecke operators.
He proved, for most squarefree $d$ in the range $26\leq d\leq 97$, that there are no non-trivial solutions to \eqref{Fermat} over $\Q(\sqrt{d})$ for $n\geq 4$. 
Kraus \cite{Kraus_2019} provided a partial resolution of 
\eqref{Fermat} over various totally real number fields of degrees $\leq 8$. 
By a partial resolution we mean for all prime exponents $n=p>B_{K}$ where 
$B_{K}$ is a constant depending only on $K$. 
For example if $K$ is a real cubic field with discriminant $148, 404$ or $564$ or if $K$ 
is the cyclic quartic field $\Q(\zeta_{16})^{+}$ then $B_{K}=5$. 
It is a natural problem then to study \eqref{Fermat} over real biquadratic fields. 
Freitas and Siksek \cite{FLTbig} initiated the study of looking at \eqref{Fermat} ``asymptotically''.
As in \cite{FLTbig}, we say Asymptotic Fermat's Last Theorem holds over $K$ if there is a constant $B_{K}$ such that there are no non-trivial solutions to \eqref{Fermat} over $K$ for prime $p>B_{K}$.
Freitas, Kraus and Siksek \cite{freitas2019class} studied the solutions to certain $S$-unit equations to prove that Asymptotic Fermat's Last Theorem holds for several infinite families of number fields -- including some real biquadratic fields. 
In this paper, we will prove the following result and discuss some obstacles that arise over more general real biquadratic fields.

\begin{thm}\label{mainthm}
Let $K=\Q(\sqrt{2},\sqrt{3})$. There are no non-trivial solutions to \eqref{Fermat} over $K$ for $n\geq 4$.
\end{thm}

We give a brief outline of the paper. 
In Section \ref{modapproach}, we apply and give a brief overview of the modular approach found in \cite{FLTbig, FLTsmall}.
In Section \ref{sec:LL}, we determine the conductor of the Frey curve 
using techniques outlined in \cite{FLTsmall}, as well as Tate's algorithm \cite[Pages 364-368]{SilvermanAdvanced}.
In Section \ref{sec:irred}, we prove that $\modpg$ is irreducible for $p\geq 13$. For $p=13$ and $17$, we prove this by studying the explicit modular parameterisation. 
For $p\geq 19$, we use work of Derickx, Kamienny, Stein and Stoll \cite{DKSS} and David \cite{DAVID2012} to get a contradiction if $\modpg$ is reducible.
In Section \ref{sec:SmExp}, we rule out solutions for certain small integer exponents.
To treat $n=9$ and $n=6$, we study the hyperelliptic curves obtained from the Fermat curve of degree $n$.
We also extend work of Mordell \cite{Mo} to determine all quartic points on the Fermat quartic lying in a quadratic extension of $\Q(\sqrt{2})$. 
In Section \ref{sec:genfields}, we give a brief overview of some obstacles that arise when extending our method to more general real biquadratic fields.
All supporting computations were performed in \texttt{Magma}, 
the scripts are available within the GitHub repository 
\url{https://github.com/MaleehaKhawaja/Fermat}.

\section*{acknowledgements}
We would like to thank Samir Siksek for several useful and enlightening discussions, 
particularly with regard to the contents of Section \ref{sec:irred} and the proof of Theorem \ref{thm:F9}. 
We would also like to thank Philippe Michaud-Jacobs, Jeremy Rouse and Michael Stoll for helpful correspondence, 
and Pedro Jos\'{e} Cazorla Garcia for a careful reading of an earlier version of this paper. 
We would like to thank the authors, Yasuhiro Ishitsuka, Tetsushi Ito and Tatsuya Ohshita, 
of \cite{Ishitsuka2019ExplicitCO} for making us aware of their work determining all points on the Fermat quartic 
lying in a quadratic extension of $\Q(\zeta_{8})$. 
We would like to thank the anonymous referees for their careful reading of the paper 
as well as the invaluable comment and suggestions they provided.
The first author thanks the University of Sheffield for their financial support 
via a doctoral training partnership scholarship (ESPRC grant no EP/T517835/1).

\section{The modular approach}
 \label{modapproach}

Let $K$ be a totally real field (until otherwise specified) and let $\mathcal{O}_{K}$ denote its ring of integers. Let $p\geq 5$ be a prime. Suppose $(a,b,c)$ is a non-trivial solution to \eqref{Fermat} over $K$ with exponent $p$. 
The traditional Frey curve associated to $(a,b,c)$ is given by
\begin{equation*}
y^{2}=x(x-a^{p})(x+b^{p}).
\end{equation*}
Our Frey curve will be a quadratic twist of this elliptic curve by a well-chosen unit 
$\varepsilon\in \OO_{K}^{\ast}$.
We write
\begin{equation}\label{eqn:Frey}
E=E_{a,b,c,\varepsilon} \; : \; y^2=x(x-\varepsilon a^p)(x+\varepsilon b^p).
\end{equation}
The reason for allowing twists by units is to reduce the number
of possibilities for the conductor of the Frey curve. 
Write $\cN_\varepsilon$ for the conductor of the Frey curve $E$ above.
We denote by $\modpg$ the mod $p$ Galois representation associated to $E$.\\

The following theorem \cite[Theorem 7]{FLTbig} is formulated from the combination of the works of Fujiwara \cite{Fujiwara}, Jarvis \cite{Jarvis2, Jarvis}, and Rajaei \cite{Rajaei}.

\begin{thm}[Freitas and Siksek]\label{levellowering}
Let $K$ be a totally real field. 
Let $p\geq 5$ be a prime. Suppose $\mathbb{Q}(\zeta_{p})^{+}\nsubseteq K$. 
Let $E$ be an elliptic curve over $K$ with conductor $\mathcal{N}$. 
Suppose $E$ is modular and $\bar{\rho}_{E,p}$ is irreducible. 
Denote by $\Delta_{\mathfrak{q}}$ the discriminant for a local minimal model of $E$ at a prime ideal $\mathfrak{q}$ of $K$. 
Let
\begin{equation*}
    \mathcal{M}_{p}:=\prod_{
    \substack{\mathfrak{q}\|\mathcal{N},\\ p \mid v_{\mathfrak{q}}(\Delta_{\mathfrak{q}})
    }
    }\mathfrak{q},\qquad \mathcal{N}_{p}:=\dfrac{\mathcal{N}}{M_{p}}.
\end{equation*} 
Suppose the following conditions are satisfied for all prime ideals $\mathfrak{q}\mid p$: 
\begin{enumerate}[(i)]
    \item $E$ is semistable at $\mathfrak{q}$;
    \item $p \mid v_{\mathfrak{q}}(\Delta_\mathfrak{q})$;
    \item the ramification index satisfies $e(\mathfrak{q}/p)<p-1$.
    \end{enumerate}
Then, $\Gisom$ where $\mathfrak{f}$ 
is a Hilbert eigenform of parallel weight $2$ that is new at level $\mathcal{N}_{p}$ and 
$\varpi$ is a prime ideal of $\mathbb{Q}_{\mathfrak{f}}$ that lies above $p$.
\end{thm}

We apply Theorem \ref{levellowering} to the Frey curve \eqref{eqn:Frey} in order to contradict the existence of the putative solution $(a,b,c)$.

Several advances have been made in the direction of establishing the modularity of elliptic curves over totally real number fields. For example, the modularity of elliptic curves over real quadratic fields \cite{freitas2014elliptic} and totally real cubic fields \cite{cubicmod} has been established. Moreover, thanks to the following result of Box \cite[Theorem 1.1]{box2021elliptic}, we now know elliptic curves over most totally real quartic fields are modular.

\begin{thm}[Box]\label{Box}
Let $K$ be a totally real quartic field not containing $\sqrt{5}$. Every elliptic curve over $K$ is modular.
\end{thm}

We turn to the question of how to show conditions $i)$ and $ii)$ of Theorem \ref{levellowering} are satisfied. 
Let $\mathcal{H}=\Cl(K)/\Cl(K)^{2}$ where $\Cl(K)$ denotes the class group of $K$. 
We can assume, without loss of generality, that any non-trivial solution $(a,b,c)$ to \eqref{Fermat} is integral. 
By Lemma 3.3 of \cite{FLTsmall}, $a,b,c$ are coprime away from a small set of primes,  i.e., $\gcd(a,b,c)=\mathfrak{m}\cdot\tau^{2}$ for some $\mathfrak{m}\in\mathcal{H}$ and odd prime ideal $\tau\neq\mathfrak{m}$. 
Consider now the following result of Freitas and Siksek \cite[Lemma 3.3]{FLTsmall} which addresses conditions $i)$ and $ii)$ above.

\begin{lem}[Freitas and Siksek]\label{semistablelem}
Let $K$ be a totally real field. Let $S$ denote the set of primes of $K$ above $2$. 
Let $\mathfrak{q}$ be a prime ideal of $K$ such that $\mathfrak{q}\notin S\cup\{\mathfrak{m}\}$. 
Then $E$ is semistable at $\mathfrak{q}$ and $p \mid v_{\mathfrak{q}}(\Delta_{\mathfrak{q}})$.
\end{lem}

We remark that our Frey curve is a quadratic twist of the usual Frey curve by 
a unit thus the set of primes dividing the conductor remains unchanged away from $2$.

\medskip

The Jacobian of the Fermat curve of degree 5, 7 and 11 has finitely many rational points. 
Since the divisor obtained from the formal sum of a point and its Galois conjugates gives a rational divisor,
this allows the study of points on these Fermat curves over fields of low degree.
Klassen and Tzermias \cite{MatthewKlassen1997} have classified all points on the Fermat quintic defined over number fields of degree at most 6. 
Building on this work, Kraus \cite{KrausQuartic} has provided an algebraic description of the quartic points on the Fermat quintic.
Tzermias \cite{Tzermias98} has determined all points on the Fermat septic defined over number fields of degree at
most 5.
Gross and Rohrlich \cite{Rohrlich1978} have determined all points on \eqref{Fermat} with exponent $p=11$ over fields of degree at most 5.
We can thus suppose that $n=4,6,9$ or $n=p\geq 13$ is a prime. \\

Throughout, unless otherwise specified, let $K=\Q(\sqrt{2},\sqrt{3})$ and let 
$\OO_{K}$ denote its ring of integers.
Let $p\geq 13$ be a prime. 
Suppose there is a non-trivial solution $(a,b,c)$ to \eqref{Fermat} over $K$ with exponent $p$. Let $E$ be the Frey curve \eqref{eqn:Frey} associated to this solution. 
By Theorem \ref{Box}, $E$ is modular. 
We note that $K$ has class number 1 and thus $\mathfrak{m}=1\cdot \mathcal{O}_{K}$. Suppose $\fq\mid p$ is a prime ideal of $K$. 
By Lemma \ref{semistablelem}, assumptions $i)$ and $ii)$ of Theorem \ref{levellowering} are satisfied for $\fq$. In particular, $E$ is semistable away from 2. 
Thus in order to prove Theorem \ref{mainthm}, it remains to:
\begin{enumerate}
	\item Determine the reduction type of $E$ at $\fP$, where $\fP$ 
    is the unique prime above $2$.
	\item Prove that $\modpg$ is irreducible, for $p\geq 13$.
	\item Eliminate the Hilbert newforms arising as a result of level lowering (Theorem \ref{levellowering}).
	\item Rule out solutions to \eqref{Fermat} for $n=4, 6$ and 9.
\end{enumerate}

\section{Computing the lowered level}\label{sec:LL}

Write $\cN_\varepsilon$ for the conductor of the Frey curve $E$ above.
We note that $2\mathcal{O}_{K}=\mathfrak{\fP}^{4}$, and $\OO_{K}/\fP=\mathds{F}_{2}$.
Thus $\fP$ divides exactly one of $a, b, c$, since $\gcd(a,b,c)=1$. 
Without loss of generality, we suppose $\fP\mid b$.

\begin{lem}\label{lem:poss}
Suppose that either $p\geq 17$ or $p=13$ and $\ord_{\fP}(b)\geq 2$.
There is some $\varepsilon \in \OO_K^*$ such that 
one of the following holds.
\begin{enumerate}[(i)]
\item Either $E$ has multiplicative reduction at $\fP$,
\item or $E$ has additive potentially multiplicative
reduction at $\fP$, and $\ord_{\fP}(\cN_\varepsilon)=4$.
\end{enumerate}
\end{lem}
\begin{proof}
Write $c_4$, $c_6$, $\Delta$ and $j$ for the usual
invariants attached to the model \eqref{eqn:Frey}.
A straightforward computation shows that
\[
c_4= \varepsilon^2 \cdot 16 \cdot (c^{2p}-a^p b^p), \qquad 
\Delta=\varepsilon^6 \cdot 16 \cdot (abc)^{2p},
\qquad j=c_4^3/\Delta.
\]
We recall that $\fP\mid b$.
Write $t=\ord_\fP(b)$. Then, 
\begin{equation}\label{eqn:j}
\ord_\fP(j)=3 \ord_\fP(c_4)-\ord_\fP(\Delta) = 32- 2 p t.
\end{equation}
Under the assumptions of the lemma, we have $\ord_\fP(j)<0$, thus we have
potentially multiplicative reduction at $\fP$ (irrespective of the choice of $\varepsilon$).

The rest of the lemma 
is a consequence of \cite[Lemma 4.4]{FLTsmall}. We give some
of the details. Let
\[
\fb = \fP^{2 \ord_\fP(2)+1} =\fP^9.
\]
Consider the natural map 
\[
\Phi : \OO_K^* \rightarrow (\OO_K/\fb)^*/((\OO_K/\fb)^*)^2.
\]
By an explicit computation in \texttt{Magma}, we find that the image of $\Phi$
has index $2$ in the codomain, and that $\lambda_1=1$
and $\lambda_2=-1+2 \mu$ are elements of $\OO_K$ which represent
the co-kernel, where $\mu=\sqrt{2}+\sqrt{3}$. 
Let $n_i=\ord_\fP(\Delta(L_i/K))$
where $L_i=K(\sqrt{\lambda_i})$ and $\Delta(L_i/K)$
is the relative discriminant ideal for the extension $L_i/K$.
We find that $n_1=0$ and $n_2=2$. 
By the aforementioned lemma, there is a unit $\varepsilon \in \OO_K^*$ such that
$\ord_\fP(\cN_\varepsilon)=1$ or $4$.
\end{proof}

In Lemma \ref{lem:poss}, we determined the conductor of the Frey curve $E$ for all primes $p\geq 17$,
and a suitable choice of $\varepsilon \in \OO_K^*$.
In particular, we prove that $E$ either has multiplicative reduction 
or additive potentially multiplicative reduction at $\fP$. 
This proof fails for $p=13$ in the case that $\ord_{\fP}(b)=1$, and we treat this case separately in the rest of the section.

\begin{lem}\label{lem:13surj}
Suppose $p=13$ and $\ord_{\fP}(b)=1$.
Then there is a unit $\varepsilon\in\OO_{K}^{\ast}$ and $\alpha\in\OO_{K}$ such that
	\[
	\fP^6\mid (\varepsilon b^{13}-\varepsilon a^{13}-\alpha^2),
	\]
	where $\fP\nmid\alpha$.
\end{lem}

\begin{proof}
Let
\[
\theta: \OO_{K}^{\ast}\rightarrow U/U^2,
\]
where $U=\big(\OO_{K}/\fP^6\big)^\ast$. 
We checked that $\theta$ is surjective using a straightforward computation in \texttt{Magma}.
Let $\beta=b^{13}-a^{13}$.
Note that $\fP\nmid\beta$.
As $\theta$ is surjective, there is some $\gamma \in \OO_K^*$
such that $\theta(\gamma)=\beta U^2$. 
Thus $\beta\equiv\gamma\alpha^2 \pmod{\fP^6}$ for some 
$\alpha\in \OO_{K}\setminus \fP$.
Let $\varepsilon=\gamma^{-1}\in\OO_{K}^{\ast}$.
Then $\varepsilon\beta\equiv \alpha^2 \pmod{\fP^6}$,
which proves the lemma.
\end{proof}

Let $\varepsilon\in\OO_K^{\ast}$ be as in Lemma \ref{lem:13surj}.
We begin by working with the Frey curve 
\begin{equation}\label{eq:Frey13}
	E_{13,\; \varepsilon}: y^2 = x(x-\varepsilon a^{13})(x+\varepsilon b^{13}).
\end{equation}
We recall that, by Lemma \ref{semistablelem}, $E_{13,\; \varepsilon}$ is 
semistable away from $\fP$. Thus in order to determine the conductor of $E_{13,\; \varepsilon}$, it remains to determine the reduction type of $E_{13,\; \varepsilon}$ at $\fP$.

\begin{lem}\label{lem:13}
Suppose $\ord_{\fP}(b)=1$. The Frey curve $E_{13,\; \varepsilon}$ has additive potentially good reduction at $\fP$. Moreover $\ord_\fP(\mathcal{N})=5$, where $\mathcal{N}$ is
 the conductor of $E_{13,\; \varepsilon}$.
\end{lem}

\begin{proof}
Let $\alpha\in\OO_{K}$ be as in Lemma \ref{lem:13surj}.
Recall that $K$ has class number $1$ and therefore
every ideal is principal. Let $\pi$ be an generator
for $\fP$. For example, we can take
\[
\pi = \frac{\mu^3 + \mu^2 - 9\mu - 9}{4},
\]
where $\mu=\sqrt{2}+\sqrt{3}$.
We make the substitutions
\[
	x\mapsto \pi^6 x,\qquad y\mapsto\alpha\pi^6 x+\pi^9 y.
\]
This yields the model
\[
	W: y^2+\frac{2\alpha }{\pi^3} xy= x^3 + \dfrac{(\varepsilon b^{13}- \varepsilon a^{13}-\alpha^2)}{\pi^6}x^2-\dfrac{\varepsilon^2 a^{13}b^{13}}{\pi^{12}}x
\]
which is integral by Lemma~\ref{lem:13surj}, and has discriminant
\[
\Delta(W)=\frac{\Delta(E_{13,\; \varepsilon})}{\pi^{36}}= \frac{16\varepsilon^6 a^{26}{b^{26}}c^{26}}{\pi^{36}}.
\]
Note that $\ord_\fP(\Delta(W))=6$. 
Thus $W$ is minimal at $\fP$.
We use 
Tate's algorithm \cite[Pages 364-368]{SilvermanAdvanced} to compute
the valuation of the conductor for $W$.
Let $a_1,\dotsc,a_6$ be the usual $a$-invariants for $W$ given in the above model, and let $b_2,\dotsc,b_8$ be the corresponding $b$-invariants:
\[
	b_{2}=\frac{4 \alpha^2}{\pi^6},\quad b_{4}=-\dfrac{2 \varepsilon^2 a^{13}b^{13}}{\pi^{12}},\quad b_{6}=0,\quad b_{8}=-\dfrac{\varepsilon^4 a^{26}b^{26}}{\pi^{24}}.
\]
In particular, $\fP \mid a_3$, $a_4$, $b_2$, and $\fP^2 \mid a_6$
and $\ord_{\fP}(b_{8})=2$. 
Thus, by Step 4 of Tate's algorithm the reduction type for $W$ at $\fP$ is III
and the valuation of the conductor at $\fP$ is
\[
\ord_{\fP}(\mathcal{N})=\ord_{\fP}(\Delta(W))-1=5.
\]
\end{proof}

\section{Proving irreducibility of $\modpg$}\label{sec:irred}

In this section, we prove that $\modpg$ is irreducible for $p\geq 13$.
In particular, we show that a possible consequence of $\modpg$ being reducible
is that $E$ has a $K$-rational point of order $p$. 
In this instance, by work of Derickx, Kamienny, Stein and Stoll \cite[Theorem 1.2]{DKSS}, we obtain a contradiction if 
$p\geq 19$. We thus treat the primes $p=13$ and $17$ separately.\\

Since the Frey curve $E$ has full 2-torsion over $K$, 
it is sufficient to show that there are no non-cuspidal $K$-rational points on one of the modular curves $X_{0}(p), X_{0}(2p)$ or $X_{0}(4p)$.
We find it convenient to work with the modular curves $X_{0}(26)$ and $X_{0}(34)$.
In particular, we show that $X_{0}(26)(K)=X_{0}(26)(\Q)$ and $X_{0}(34)(K)=X_{0}(34)(\Q)$.
All points in $X_{0}(26)(\Q)$ and $X_{0}(34)(\Q)$ are cuspidal thus proving the irreducibility of $\modpg$ for $p=13$ and $17$.

\subsection{$p=13$} 
Using the explicit modular parametrisation, we prove that if $P\in X_{0}(26)(K)$ then either $P\in X_{0}(26)(\Q(\sqrt{3}))$ or that $C(\Q(\sqrt{3}))$ is non-empty, where $C$ is a genus 2 hyperelliptic curve.
In the first case, by work of Bruin and Najman \cite{Bruin2015}, $P\in X_{0}(26)(\Q)$. In the second case, we show that $C(\Q(\sqrt{3}))$ is empty.

\begin{lem}
$\bar{\rho}_{E, 13}$ is irreducible.
\end{lem}

\begin{proof}
	We prove that $X_{0}(26)(K)=X_{0}(26)(\Q)$. 
	We work with the model
\begin{equation}
 \label{X026}
    X_{0}(26): y^2 = x^6 -8x^5 + 8x^4 - 18x^3 + 8x^2 -8x +1
\end{equation}
given in \texttt{Magma}. 
Let
\begin{equation*}
    E': y^{2} + xy + y = x^{3} - x^{2} - 3x + 3.
\end{equation*}
Then $E'$ is the elliptic curve with Cremona label \texttt{26b1}. 
Suppose $P=(a,b)\in X_{0}(26)(K)$. 
Note that if $a=1$ then $b^2=-16$, i.e., $P\notin X_{0}(26)(K)$. 
Suppose from now on that $a\neq 1$. 
Using \texttt{Magma}, we find the explicit parametrisation
\begin{align*}
    \pi: X_{0}(26)&\longrightarrow E'\\
    (a,b)&\longmapsto \left(-\dfrac{(a+1)^2}{(a-1)^2},\; \dfrac{-2b+2a(a-1)}{(a-1)^3}\right).
\end{align*}
Let $L=\Q(\sqrt{3})$. 
We checked using \texttt{Magma} that $E'(K)=E'(L)$. 
It immediately follows that
\begin{equation*}
    \left(\dfrac{a+1}{a-1}\right)^2\in L.
\end{equation*}
Let 
\[
\sigma: K\rightarrow K,\quad \sigma(\sqrt{2})=-\sqrt{2},\quad \sigma(\sqrt{3})=\sqrt{3}.
\]
Then 
\[
	\sigma\left(\dfrac{a+1}{a-1}\right)=\dfrac{a+1}{a-1}\quad\text{ or }\quad 
	\sigma\left(\dfrac{a+1}{a-1}\right)=-\dfrac{a+1}{a-1}.
\]
Thus there are two cases to consider:
\begin{enumerate}
	\item[(1)] $(a+1)/(a-1)\in L$;
	\item[(2)] $(a+1)/(a-1)\in\; \sqrt{2}\cdot L$. 
\end{enumerate}

\noindent \textbf{Case (1)} In this case $a\in L$, and it immediately follows from 
the parameterisation of $\pi$ that $b\in L$. 
Observe that $X_{0}(26)$ has infinitely many quadratic points of the form 
$(r, \sqrt{f(r)})$, where $r\in\Q$. 
Such points are called \textbf{non-exceptional} 
and all other quadratic points are called \textbf{exceptional}. 

\noindent \textbf{Case (1.1)} 
If $a\in L\setminus \Q$ then $P$ is an exceptional 
quadratic point on $X_{0}(26)$. Bruin and Najman \cite[Table 3]{Bruin2015} have given an explicit
description of all quadratic points on $X_{0}(26)$. 
They find that all exceptional quadratic points are 
defined over $\Q(\sqrt{d})$ for $d=-1, -3, -11$ and $-23$. 

\noindent \textbf{Case (1.2)}
If $a\in\Q$ then $b^2\in\Q$. 
Then $P$ is a non-exceptional quadratic point on 
$X_{0}(26)$ defined over $L$. 
Moreover $P$ corresponds to a rational point 
on the quadratic twist of $X_{0}(26)$ over $\Q(\sqrt{3})$. 
We denote this quadratic twist by $X_{3}$. 
We checked using \texttt{Magma} that $X_{3}$ has 
no points defined over $\Q_{3}$. 
Thus $X_{3}(\Q)$ is empty.

\noindent \textbf{Case (2)} In this case we have 
$(a+1)/(a-1)\in\sqrt{2}\cdot L$, i.e.,
\begin{equation}
 \label{alphafrac}
    \dfrac{a+1}{a-1}=\sqrt{2}\alpha,\; \text{ for some } \alpha\in L.
\end{equation}
Note the following identity:
\begin{equation}
  \label{sqrt3idnty}
    \left(\dfrac{a+1}{a-1}\right)^2 -1=\dfrac{(a+1)^2-(a-1)^2}{(a-1)^2}=\dfrac{4a}{(a-1)^2}=\dfrac{4a(a-1)}{(a-1)^3}.
\end{equation}
From the parametrisation of $\pi$ and \eqref{sqrt3idnty}, we see that
\begin{equation*}
    \dfrac{b}{(a-1)^3}\in L.
\end{equation*}
Note the following identity
\begin{multline}
 \label{a-1div}
    16\left(\dfrac{a^6-8a^5+8a^4-18a^3+8a^2-8a+1}{(a-1)^6}\right)
    \\=-4\left(\dfrac{a+1}{a-1}\right)^6-3\left(\dfrac{a+1}{a-1}\right)^4+10\left(\dfrac{a+1}{a-1}\right)^2+13.
\end{multline}
By combining \eqref{X026} and \eqref{a-1div}, we obtain
\begin{equation*}
    \left(\frac{4b}{(a-1)^3}\right)^2=-4\left(\dfrac{a+1}{a-1}\right)^6-3\left(\dfrac{a+1}{a-1}\right)^4+10\left(\dfrac{a+1}{a-1}\right)^2+13.
\end{equation*}
After making the substitutions $\beta=4b/(a-1)^3$ and \eqref{alphafrac}, we obtain 
\begin{equation*}
    \beta^2=-32\alpha^6-12\alpha^4+20\alpha^2+13.
\end{equation*}
Thus $(\alpha, \beta)$ is a $L$-rational point on the curve
\begin{equation*}
C: y^2=-32x^6-12x^4+20x^2+13.
\end{equation*}
Write $\mathcal{O}_{L}$ for the ring of integers 
of $L$. Then $13\mathcal{O}_{L}=\mathfrak{p}_{1}\mathfrak{p}_{2}$. 
We checked using \texttt{Magma} that there there are no points on $C$ defined over the completion of 
$L$ at $\mathfrak{p}_{1}$. Thus $C(L)$ is empty.
\end{proof}

\subsection{$p=17$} Using the explicit modular parametrisation, we show that if $P\in X_{0}(34)(K)$ then either $P\in X_{0}(34)(\Q(\sqrt{2}))$ or that $C(\Q(\sqrt{2}))$ is non-empty, where $C$ is the quadratic twist of $X_{0}(34)$ over $\Q(\sqrt{3})$. 
In the first case, by Ozman and Siksek \cite{ozman2018quadratic}, $P\in X_{0}(34)(\Q)$. 
In the second case, we show that $C(\Q(\sqrt{2}))$ is empty.

\begin{lem}
$\bar{\rho}_{E, 17}$ is irreducible.
\end{lem}

\begin{proof}
We prove that $X_{0}(34)(K)=X_{0}(34)(\Q)$. 
We work with the model
\begin{equation}
  \label{X034}
    X_{0}(34):x^4-y^4+x^3+3xy^2-2x^2+x+1=0
\end{equation}
given in \texttt{Magma}.
Making the change of variables $x\mapsto x,\; y\mapsto y^2$ yields the curve
\begin{equation*}
    C': x^4-y^2+x^3+3xy-2x^2+x+1=0.
\end{equation*}
Since $C'$ has genus 1, we can transform it to an elliptic curve
\begin{equation}
      \label{X034map}
        C' \rightarrow E^{\prime},\qquad (x,y) \mapsto (2(x^2-2x+y),\; 4x(x^2-2x+y))
\end{equation}
where
\begin{equation*}
    E': y^2+xy+2y=x^3-4x
\end{equation*}
is the elliptic curve with Cremona label \texttt{34a1}. We deduce the explicit modular parametrisation
\begin{align*}
	 \pi: X_{0}(34) &\longrightarrow E'\\
	  (x,y) &\longmapsto (2(x^2-2x+y^2), 4x(x^2-2x+y^2)).
\end{align*}
Let $L=\Q(\sqrt{2})$. 
Using \texttt{Magma} we find that $E'(K)=E'(L)$. 
Suppose $P=(a,b)\in X_{0}(34)(K)$. 
Since
\begin{equation*}
    2(a^2-2a+b^2),\qquad 4a(a^2-2a+b^2)\in L,
\end{equation*}
it follows that either $a^2-2a+b^2=0$ or $a\in L$. Suppose the former is true, i.e., 
\begin{equation}
 \label{X034nonzero}
    b^2=2a-a^2.
\end{equation}
We substitute \eqref{X034nonzero} into \eqref{X034} to find that
\begin{equation*}
    2a^3+a+1=0,
\end{equation*}
and $a\not\in K$. 
Thus, $a\in L$ and hence $b^2\in L$. 
Either $b\in L$ or $b=\sqrt{3}\beta$ for some $\beta\in L$. 
If $b\in L$ then $P\in X_{0}(34)(L)$. 
Ozman and Siksek \cite{ozman2018quadratic} have determined the quadratic points on $X_{0}(34)$, and they found that there are no real quadratic points on $X_{0}(34)$. Thus $P\in X_{0}(34)(\Q)$.

Suppose $b=\sqrt{3}\beta$ for some $\beta\in L$. Thus, $(a,\beta)$ is an $L$-rational point on the curve
\begin{equation*}
 \label{X034twist}
    C:\; x^4-9y^4+x^{3}+9xy^2-2x^2+x+1=0,
\end{equation*}
where $C$ is the quadratic twist of $X_{0}(34)$ over $\Q(\sqrt{3})$. 
Note that 3 is inert in $L$.
We checked using \texttt{Magma} that there there are no points on $C$ defined over the completion of $L$ at $3\OO_{L}$. Thus $C(L)$ is empty.
\end{proof}

\subsection{$p\geq 19$}

We let $E=E_{a,b,c,\varepsilon}$ where $\varepsilon \in \OO_K^*$ is chosen so that one of the two possibilities in Lemma~\ref{lem:poss} hold. 
Suppose $\overline{\rho}_{E,p}$ is reducible. Then
\[
\overline{\rho}_{E,p} \sim \begin{pmatrix}
\theta & * \\
0 & \theta^\prime
\end{pmatrix}
\]
where $\theta$, $\theta^\prime$ are characters $G_K \rightarrow \F_p^*$. 
Recall that $\chi_p=\det(\overline{\rho}_{E,p})=\theta \theta^\prime$
where $\chi_p$ denotes the mod $p$ cyclotomic character. 
We let $\cN_{\theta}$ and $\cN_{\theta^\prime}$ denote the conductors of $\theta$ and $\theta^\prime$, respectively. 
We shall require the following result of Freitas and Siksek \cite[Lemma 6.3]{FLTsmall}.

\begin{lem}[Freitas and Siksek]\label{lem:conductor}
    Let $E$ be an elliptic curve defined over a number field $K$ with conductor $\cN$.
    Let $p\geq 5$ be a prime, and let $\mathfrak{q}\nmid p$ be a prime.
    Let $\theta$ and $\theta^\prime$ be as above.
    If $\modpg$ is reducible then 
    \[
\ord_{\mathfrak{q}}(\cN_{\theta})=\ord_{\mathfrak{q}}(\cN_{\theta^\prime})=\begin{cases}
0 & \text{if $E$ has good or multiplicative reduction at $\mathfrak{q}$}\\
\dfrac{\ord_{\mathfrak{q}}(\cN)}{2}\in\Z & \text{if $E$ has additive reduction at $\mathfrak{q}$}
\end{cases}
\]
\end{lem}

\begin{lem} 
Let $p \ge 19$. Then $\overline{\rho}_{E,p}$ is irreducible.
\end{lem}
\begin{proof}

Suppose $\modpg$ is reducible. 
Since $p$ is unramified in $K$, and $E$ has good
or multiplicative reduction at $\fp \mid p$, we have that for any $\fp \mid p$, precisely one of $\theta$, $\theta^\prime$
is ramified at $\fp$ (see Kraus \cite[Lemma 1]{KRAUSQUADRATIC}). 

Let us first suppose that either of $\theta$, $\theta^\prime$ is unramified at all $\fp \mid p$ (and thus the other is ramified at all $\fp \mid p$). 
We note that replacing $E$ by a $p$-isogenous
elliptic curve, if necessary, allows us to swap $\theta$ and $\theta^\prime$. Thus we may suppose that $\theta$ is unramified at all the primes above $p$ and hence $\theta$ is unramified away from $\fP$.

We shall use Lemma \ref{lem:conductor} to determine $\cN_{\theta}$. 
Suppose we are in case (i) of 
Lemma \ref{lem:poss}, and $E$ has multiplicative reduction at $\fP$.
Then by Lemma \ref{lem:conductor}, we
have 
$\ord_\fP(\cN_{\theta^\prime})
=\ord_\fP(\cN_{\theta})=0$.
Suppose now that we are in case (ii)
of Lemma \ref{lem:poss}, and $E$ has additive reduction at $\fP$.
Then by Lemma \ref{lem:conductor}, we have
\[
\ord_{\fP}(\cN_\theta)=\ord_{\fp}(\cN_{\theta^\prime})=\frac{1}{2} \ord_\fP(\cN_\varepsilon)=2.
\]
Hence either $\cN_\theta=1$ or $\fP^2$. 
Let $\infty_1,\dotsc,\infty_4$
denote the four real places of $K$. Then $\theta$ is a character
for the ray class group of the modulus $\infty_1 \cdots \infty_4$
in the first case, and of the modulus $\fP^2 \cdot \infty_1 \cdots \infty_4$
in the second case. Using \texttt{Magma} we find that
this ray class group is $\Z/2\Z$ in either case. Thus, the order of $\theta$ divides 2, and $\theta$ is either trivial or a quadratic character. In the first case, when $\theta$ is trivial, $E$ has a $K$-rational point of order $p$.
In the second case, let $E^\prime$ be the quadratic twist of
$E$ by $\theta$. 
Then 
\[
\overline{\rho}_{E^\prime,p} \sim \begin{pmatrix}
\theta^2 & * \\
0 & \theta\theta^\prime
\end{pmatrix}
=
\begin{pmatrix}
1 & * \\
0 & \chi_{p}
\end{pmatrix}.
\]
Thus $E^\prime$ has a $K$-rational point of order $p$.
In both cases, we obtain an elliptic curve with a point of order $p$ defined over $K$ (a quartic field). By Derickx, Kamienny, Stein and Stoll \cite[Theorem 1.2]{DKSS}, $p\leq 17$. We obtain a contradiction since $p\geq 19$.

Fix $\fp_0$ to be a prime ideal of $\OO_K$ above $p$. Let $G=\Gal(K/\Q)$.
Then $G$ acts transitively on the primes $\fp \mid p$. Let $S$ be
the set of $\tau \in G$ such that $\theta$ is ramified at $\tau(\fp_0)$. 
We know from the above that $S$ is a proper subset of $G$, i.e., $S \ne \emptyset$ and $S \ne G$.
For a prime ideal $\fq$ of $\OO_K$ we write $I_\fq$ 
for an inertia subgroup of $G_K$ corresponding to $\fq$.
Thus $\theta \vert_{I_\fq}=1$ for all 
\[
\fq \notin \{\fP\} \cup \{\tau(\fp_0) : \tau \in S\}.
\]
By Lemma \ref{lem:poss}, $E$ has potentially multiplicative reduction
at $\fP$. By the theory of the Tate curve \cite[Proposition 1.2]{DAVID2012}
$\theta^2 \vert_{I_\fP}=1$. Let $\phi=\theta^2$. Then
$\phi \vert_{I_\fq}=1$ for all 
\[
\fq \notin \{ \tau(\fp_0) : \tau \in S\}.
\]
Recall that $\theta^\prime$ is unramified at $\fq \in \{ \tau(\fp_0) : \tau \in S\}$. Since $\theta \theta^\prime=\chi_p$, we conclude that
\begin{equation}\label{eqn:phi}
\phi \vert_{I_\fq} =
\begin{cases}
\chi_p^2 \vert_{I_\fq} & \fq \in \{ \tau(\fp_0) : \tau \in S\} \\
1 & \text{otherwise}.
\end{cases}
\end{equation}

Let $u \in \OO_K^*$. 
We define the twisted norm of $u$ attached to $S$ to be
\[
\fN_S(u) = \prod_{\tau \in S} (\tau(u))^2.
\]
By the proof of \cite[Proposition 2.6]{DAVID2012}, the existence of $\phi$ satisfying \eqref{eqn:phi}
ensures that
\[
\fp_0 \mid (\fN_S(u)-1).
\]
Let $\mu=\sqrt{2}+\sqrt{3}$, and let
\[
u_1=\mu,\; u_2=(-\mu^3 + 9\mu + 2)/2,\; u_3=(\mu^3 - \mu^2 - 9\mu + 5)/4;
\]
this is a basis for $\OO_K^*/\{\pm 1\}$.
Then, $p \mid B_S$ where
\[
B_S= \Norm \left(\sum_{i=1}^3 (\fN_S(u_i)-1) \cdot \OO_K \right).
\]
We used \texttt{Magma} to compute $B_S$ for all non-empty proper subsets $S$ of $G=\Gal(K/\Q)$.
In all cases we found that if $p\mid B_S$ then $p=2$ or $3$.
Thus we obtain a contradiction.
\end{proof}

\section{Eliminating Hilbert newforms}
Let
\[
\cN_0=\begin{cases}
\fP & \text{if we are in case (i) of Lemma~\ref{lem:poss}}\\
\fP^4 & \text{if we are in case (ii) of Lemma~\ref{lem:poss}}\\
\fP^5 & \text{if } p=13 \text{ and } \ord_{\fP}(b)=1.
\end{cases}
\]
Applying level lowering (i.e. Theorem \ref{levellowering}) we obtain
\[
\overline{\rho}_{E,p} \sim \overline{\rho}_{\ff,\fp}
\]
where $\ff$ is a Hilbert newform of parallel weight $2$
and level $\cN_0$, and $\fp$ is some prime above $p$
in $\Q_\ff$, the Hecke eigenvalue field of $\ff$.
Using \texttt{Magma} we find that there are no newforms
with parallel weight $2$ and level $\fP$ or level $\fP^5$, obtaining
a contradiction in these cases.\\

We thus suppose we are in case (ii) of Lemma~\ref{lem:poss}. 
For the level $\fP^4$
we find that there are two newforms $\ff_1$, $\ff_2$
and for both the corresponding Hecke eigenvalue field
is $\Q$. Let $E_1/K$, $E_2/K$ be the following elliptic curves:
\[
E_1: y^2 + (\mu + 1)xy = x^3 + \dfrac{1}{4}(-\mu^3 - \mu^2 - 3\mu + 
    5)x^2 + \dfrac{1}{2}(-\mu^3 - 5\mu)x + \dfrac{1}{4}(\mu^3 + 7\mu^2 - 9\mu - 3)
\]
\[
E_2: y^2 + \dfrac{1}{4}(\mu^3 + \mu^2 + 3\mu + 3)y = x^3 + 
\dfrac{1}{2}(-\mu^2 - 1)x^2 + \mu^2x + \dfrac{1}{4}(-3\mu^3 - 17\mu^2 - \mu + 1),
\]
where $\mu=\sqrt{2}+\sqrt{3}$.
These elliptic curves have conductors $\fP^4$ and were found using the \texttt{Magma}
command \texttt{EllipticCurveSearch}. These are non-isogenous
as $a_\fq(E_1)=6$ and $a_\fq(E_2)=-6$ where $3\mathcal{O}_{K}=\fq^2$. 
By the work of Box \cite{box2021elliptic}, $E_1$, $E_2$ are modular and thus
correspond to the two Hilbert newforms
$\ff_1$, $\ff_2$ of parallel weight $2$ and level $\fP^4$.
Thus $\overline{\rho}_{E,p} \sim \overline{\rho}_{E_i,p}$
where $i=1$ or $2$. To obtain a contradiction we shall use
a standard image of inertia argument (see \cite[Lemma 3.5]{FLTbig}).

Let $j$ be the $j$-invariant of the Frey curve $E$.
By \eqref{eqn:j} we have $\ord_{\fP}(j)<0$ and $p \nmid \ord_{\fP}(j)$.
Thus, $p \mid \# \overline{\rho}_{E,p}(I_\fP)$ \cite[Proposition 6.1, Chapter 5]{SilvermanAdvanced}.
However, we find that $E_1$, $E_2$ have $j$-invariants
\[
	j_1=0\quad \text{ and }\quad j_2=-853632\mu^3 + 7682688\mu + 2417472,
\]
respectively. 
As $\ord_{\fP}(j_i) \ge 0$, we have that $E_1$, $E_2$ have potentially good reduction at $\fP$.
It follows that $\# \overline{\rho}_{E_i,p}(I_\fP) \mid 24$ from the work of Kraus \cite[Introduction]{Kraus1990}. 
As $\overline{\rho}_{E,p} \sim \overline{\rho}_{E_i,p}$, for $i=1$ or $2$, we obtain $p \mid 24$ giving a contradiction.

\section{Small Integer Exponents}\label{sec:SmExp}

We have thus far shown that there are no solutions to~\eqref{Fermat} over $K$ for prime $p\geq 5$. 
In order to complete the proof of Theorem~\ref{mainthm}, it remains to rule out solutions to~\eqref{Fermat}
for $n= 4, 6, 9$. 
We note in passing that there are non-trivial solutions to the Fermat cubic over $\Q(\sqrt{2})$ (see \cite[page 184]{JarvisMeekin}).

\subsection{$n=9$}
We are very grateful to Samir Siksek for the lengthy discussions and ideas that resulted in this proof. 
We first convert the problem of finding $K$-points on the Fermat curve of degree 9 to finding the $\Q(\sqrt{3})$-points 
on a certain hyperelliptic curve $C$. We then study the Jacobian of $C$ to show that $C(\Q(\sqrt{3}))=\{\infty\}$, 
where $\infty$ denotes the point at infinity on $C$.

\begin{thm}\label{thm:F9}
There are no non-trivial solutions to \eqref{Fermat} over $K$ for $n=9$.
\end{thm}
We find it convenient to let
\[
F_9 \; : \; x^9+y^9+z^9=0.
\]
That is, $F_{9}$ is the Fermat curve of degree 9. 
We recall that $2\mathcal{O}_{K}=\fP^4$ and that $K$ has class number $1$. 
We will prove that $F_{9}(K)=\{ (1:-1:0),\; (1:0:-1), \; (0:1:-1)\}$, i.e, $F_{9}(K)$ consists only of trivial solutions. 
Suppose $(\alpha: \beta: \gamma)\in F_{9}(K)$ is a non-trivial solution. 
We may suppose that $\alpha$, $\beta$, $\gamma \in \OO_K$ and that they are coprime. We recall that $\OO_K/\fP=\F_2$ and
\[
F_9(\F_2)=\{ (1:1:0),\; (1:0:1), \; (0:1:1)\}.
\]
Hence, by permuting $\alpha$, $\beta$, $\gamma$ appropriately, we may suppose $(\alpha:\beta:\gamma) \equiv (1:1:0) \pmod{\fP}$. 
Thus
\begin{equation}\label{eqn:fp}
\fP \mid \gamma, \qquad \fP \nmid \alpha \beta.
\end{equation}
Observe 
\[
\gamma^{18}-(\alpha^9-\beta^9)^2 \; = \; (\alpha^9+\beta^9)^2-(\alpha^9-\beta^9)^2 \; =\; 4(\alpha\beta)^9.
\]
After making the substitutions
\begin{equation}\label{eqn:uv}
u=\frac{\alpha \beta}{\gamma^2}, \qquad v=\frac{\alpha^9-\beta^9}{\gamma^9},
\end{equation}
we see that $Q_1=(u,v) \in C_1(K)$ where
\[
C_1 \; : \; y^2=-4x^9+1.
\]
Let
\[
E_1 \; : \; y^2=4x^3+1.
\]
This is an elliptic curve. Let 
\[
\pi_{1} : C_1 \rightarrow E_1, \qquad (x,y) \mapsto (-x^3,y).
\]
The elliptic curve $E_1$ has minimal Weierstrass model
\[
E_1^\prime \; : \; z^2+z=x^3
\]
which is obtained from $E_1$ by the substitution $y=2z+1$. This has Cremona label \texttt{27a3}. In particular $E_1^\prime$ has good reduction away from $3$. Let $R_1=\pi_{1}(Q_1)=(-u^3,v) \in E_1(K)$. Then $R_1$ corresponds to the point
\[
S_1=(-u^3,(v-1)/2) \in E_1^\prime(K).
\]
Let $\sigma : K \rightarrow K$ be the automorphism satisfying
\[
\sigma(\sqrt{2})=-\sqrt{2}, \qquad \sigma(\sqrt{3})=\sqrt{3}.
\]
We note that the fixed field of $\sigma$ is $L=\Q(\sqrt{3})$. Thus $S_1+S_1^\sigma \in E_1^\prime(L)$.
We checked using \texttt{Magma} that $E_1^\prime$ has rank $0$ over $L$,
and indeed
\begin{equation}\label{eqn:EdL}
E_1^\prime(L) \; =\; \{\mathcal{O}, (0,0), (0,-1)\} \; \cong \; \Z/3\Z.
\end{equation}
Thus $S_1+S_1^\sigma$ is one of these three points.
However, $\ord_\fP(u) <0$ by \eqref{eqn:fp} and \eqref{eqn:uv}. 
It follows that
\[
S_1 \equiv \mathcal{O} \pmod{\fP}.
\]
Hence
\[
S_1^\sigma \equiv \mathcal{O}^\sigma = \mathcal{O} \pmod{\fP^\sigma}.
\]
However, $\fP$ is a totally ramified prime, so $\fP^\sigma=\fP$. Thus $S_1^\sigma \equiv \mathcal{O} \pmod{\fP}$, and
\[
S_1+S_1^\sigma \equiv \mathcal{O} \pmod{\fP}.
\]
By \eqref{eqn:EdL} and the injectivity of torsion
upon reduction modulo primes of good reduction (see \cite[Appendix]{Katz}), we conclude that
\[
S_1+S_1^\sigma=\mathcal{O}.
\]
Hence
\[
R_1+R_1^\sigma=\mathcal{O}.
\]
Hence
\[
(-u^3)^\sigma = -u^3, \qquad v^\sigma=-v.
\]
As the only cube root of $1$ in $K$ is $1$, we have $u^\sigma=u$ and so $u \in L$. 
Moreover, $v^2 = -4 u^9+1 \in L$ and $v^\sigma=-v$, so $v=w/\sqrt{2}$ where $w \in L$. 
Hence $(u,w) \in C(L)$ where
\[
C \; : \; y^2=2(-4x^9+1). 
\]
\begin{lem}\label{lem:CL}
$C(L)=\{ \infty\}$.
\end{lem}
Since $u=\alpha \beta/\gamma^2$, Theorem~\ref{thm:F9} follows from
Lemma~\ref{lem:CL}. We now prove Lemma~\ref{lem:CL} through studying
$J(\Q)$ where $J$ is the Jacobian of $C$.
\begin{proof}
Let
\[
E \; : \; y^2=x^3+2,
\]
which is elliptic curve with Cremona label \texttt{1728a1}.
Let
\begin{equation}\label{eqn:pi}
\pi : C \rightarrow E, \qquad (x,y) \mapsto (-2x^3,y).
\end{equation}
Using \texttt{Magma} we find that $E$ has zero torsion and rank $1$ over $\Q$, and that in fact,
\[
E(\Q)=\Z \cdot (-1,1).
\]
We write $\Pic^0(E)$ for the group of rational degree $0$ divisors on $E/\Q$
modulo linear equivalence, and $\Pic^0(C)$ for the group of rational
degree $0$ divisors on $C/\Q$ modulo linear equivalence. We recall the standard isomorphism \cite[Proposition III.3.4]{SilvermanAEC}
\begin{equation} \label{eqn:iso}
E(\Q) \cong \Pic^0(E), \qquad P \mapsto [P-\infty],
\end{equation}
where $[D]$ denotes the linear equivalence class of a divisor $D$.
Thus
\[
\Pic^0(E) = \Z \cdot \mathcal{Q}, \qquad \mathcal{Q}=[(-1,1)-\infty].
\]

We also recall the standard isomorphism
$J(\Q) \cong \Pic^0(C)$, and we will represent elements of the Mordell--Weil group $J(\Q)$
as elements of $\Pic^0(C)$. 
Using \texttt{Magma} we find that $J$ has good reduction 
away from the primes $2$ and $3$.
Moreover a straightforward calculation in \texttt{Magma} returns that
\[
J(\F_5) \cong \Z/6\Z \times \Z/126\Z, \qquad J(\F_{13}) \cong \Z/42997\Z.
\]
As these two groups have coprime orders we conclude that $J$ has trivial torsion over $\Q$.
Moreover using \texttt{Magma} we find that $J$ has $2$-Selmer rank $1$ over $\Q$, so $J$ has rank at most $1$ over $\Q$.
The morphism $\pi$ in \eqref{eqn:pi} has degree $3$ and induces homomorphisms (see \cite[Section II.3]{SilvermanAEC})
\[
\pi_* \; : \; \Pic^0(C) \rightarrow \Pic^0(E), \qquad \left[ \sum a_i P_i \right] \mapsto \left[\sum a_i \pi(P_i) \right],
\]
and
\[
\pi^* \; : \; \Pic^0(E) \rightarrow \Pic^0(C), \qquad \left[\sum b_j Q_j \right] \mapsto \left[\sum b_j \sum_{P \in \pi^{-1}(Q_j)} e_\pi(P) \cdot P\right]
\]
where $e_\pi(P)$ denotes the ramification degree of $\pi$ at $P$.

Let
\[
\mathcal{P}=\pi^*(\mathcal{Q}) \; = \;  [ (1/\sqrt[3]{2},1)+(\omega/\sqrt[3]{2},1)+(\omega^2/\sqrt[3]{2},1)   -3 \infty] \in \Pic^0(C) \cong J(\Q).
\]
where $\omega$ is a primitive cube root of $1$.
The point $\mathcal{P}$ has infinite order on $J(\Q)$. Thus $J$ has rank exactly $1$ over $\Q$ and no torsion.
Therefore $J(\Q)=\Z \cdot \mathcal{P}^\prime$, for some $\mathcal{P}^\prime \in J(\Q)=\Pic^0(C)$. Hence
\[
\mathcal{P} = k \mathcal{P}^\prime
\]
where $k$ is a non-zero integer. Applying $\pi_*$ to both sides we obtain 
\[
k \pi_*(\mathcal{P}^\prime) = \pi_*(\mathcal{P}) = 3 \mathcal{Q}.
\]
However,  $\pi_*(\mathcal{P}^\prime) \in \Pic^0(E)=\Z \mathcal{Q}$, so 
\[
\pi_*(\mathcal{P}^\prime) = \ell \cdot \mathcal{Q}
\]
for some $\ell \in \Z$. Hence $k\ell=3$, so $k = \pm 1$ or $\pm 3$.
Using \texttt{Magma} we checked that the image of $\mathcal{P}$ under the composition
\[
J(\Q) \rightarrow J(\F_5) \rightarrow J(\F_5)/3 J(\F_5)
\] 
is non-zero. Thus $k \ne \pm 3$, so $k=\pm 1$, hence 
\[
J(\Q)=\Pic^0(C) = \Z \cdot \mathcal{P}.
\]

Suppose $P \in C(L)$. Let $\tau : L \rightarrow L$ be
the non-trivial automorphism. Then $[P+P^\tau -2 \infty] \in \Pic^0(C)$. 
Thus 
\[
[P+P^\tau - 2 \infty]  \; = \;  n \cdot \mathcal{P} \; = \; n \cdot \pi^*(\mathcal{Q}) \; = \; \pi^*(n \cdot \mathcal{Q})
\]
for some integer $n$. We claim that $n=0$. Suppose otherwise, then
$n \cdot \mathcal{Q} \in \Pic^0(E) \setminus \{0\}$
and by the isomorphism in \eqref{eqn:iso} we have $n \cdot \mathcal{Q}=[Q-\infty]$ where $Q \in E(\Q) \setminus \{\mathcal{O}\}$. Write $Q=(a,b) \in E(\Q)$ with $a$, $b \in \Q$.
Then 
\[
[P+P^\tau-2 \infty] \; = \; \pi^*([(a,b) - \infty]) \; =\; [D-3 \infty]
\]
where 
\[
D=P_1+P_2+P_3, \qquad P_j \; = \; \left( -\omega^{j-1} \sqrt[3]{a/2} \, , \,  b\right) , \qquad j=1,2,3.
\]
Hence
\[
D \;  \sim \; D^\prime, \qquad D^\prime=P+P^\tau+\infty
\]
where $\sim$ denotes linear equivalence on $C$. Write $\lvert D \rvert$ for the complete linear
system of effective divisors of $C$ linearly equivalent to $D$.
Let $r(D)=\dim \lvert D \rvert$. Note that $D^\prime \in \lvert D \rvert$ and $D^\prime \neq D$,
therefore $r(D) \ge 1$. By Riemann--Roch \cite[page 13]{ACGH},
\[
r(D) - i(D) \; = \; \deg(D) - g =-1,
\]
where $i(D) \ge 0$ is the so called index of speciality of $D$, and $g=4$ is the genus of $C$.
It follows that $i(D) >0$, and therefore that $D$ is a special divisor. By Clifford's theorem \cite[Theorem IV.5.4]{Hartshorne},
\[
r(D) \; \le \; \frac{\deg(D)}{2}=\frac{3}{2}.
\]
Hence $r(D)=1$. Thus the complete linear system $\lvert D \rvert$ is a $g^1_3$. As $C$ is hyperelliptic, by \cite[page 13]{ACGH}, $\lvert D \rvert=g^1_2+p$ where $p$ is a fixed base point of the linear system. In particular, every divisor in $\lvert D \rvert$ is the sum of $p$ and two points interchanged by the hyperelliptic involution. We apply this to $D$ itself. Thus two of the points $P_1$, $P_2$, $P_3$ are interchanged by the hyperelliptic involution. However, they all have the same $y$-coordinate $b$, so $b=0$. But $(a,b) \in E(\Q)$, so $a \in \Q$ and $a^3=-2$ giving a contradiction. 
Hence $n=0$, and so
\[
P+P^\tau \sim 2 \infty.
\]
Thus $P$, $P^\tau$ are interchanged by the hyperelliptic involution. We recall that we want to show that $P=\infty$. Suppose otherwise.
Then we can write $P=(c,d)$ where $c$, $d \in L$ and $c^\tau=c$, $d^\tau=-d$. Thus $c \in \Q$,
and $d=e/ \sqrt{3}$ with $e \in \Q$. Thus $P^\prime=(c,e) \in C^\prime(\Q)$ where 
\[
C^\prime \; : \; y^2=6(-4x^9+1).
\]
Let $J^\prime$ be the Jacobian of $C^\prime$, and
\[
E^\prime \; : \; y^2=6(4x^3+1).
\]
Using \texttt{Magma} we find that $E^\prime(\Q)=\Z \cdot (1/2,3)$. Let $\cQ^\prime=[(1/2,3)-\infty] \in \Pic^0(E^\prime)$,
so $\Pic^0(E^\prime)=\Z \cdot \cQ$. 
Let 
\[
\pi^\prime : C^\prime \rightarrow E^\prime, \qquad (x,y) \mapsto (-x^3,y).
\]
Using \texttt{Magma} we find that $J^\prime$ has trivial torsion and $2$-Selmer rank $1$,
and following the same steps as before show that $J^\prime(\Q)=\Pic^0(C)=\Z \cdot \cP^\prime$,
where $\cP^\prime=(\pi^\prime)^* (\cQ)$. Now $[P^\prime-\infty]=n \cP^\prime$
where $n$ is an integer, and must be non-zero as $P^\prime \ne \infty$.
Let $(f,g)=n \cdot (1/2,3) \in E^\prime(\Q) \setminus \{ \mathcal{O}\}$.
As before, we find that 
\[
P^\prime+2\infty \sim P_1^\prime+P_2^\prime+P_3^\prime, \qquad P_j^\prime = ( -\omega^{j-1} \cdot \sqrt[3]{f} \, ,\, g).
\]
Following the same steps as before, it follows that $g=0$, so $f^3=-1/4$ contradicting $f \in \Q$.
We can thus conclude that if $P\in C(L)$ then $P=\infty$. This completes the proof of Lemma~\ref{lem:CL} and therefore Theorem~\ref{thm:F9}. 
\end{proof}

\subsection{$n=6$}
We show that a $K$-point on the Fermat curve of degree 6 induces a $K$-point $P$ on a certain hyperelliptic curve $C$. 
Let $E$ be the elliptic curve obtained by taking the quotient of $C$ by a certain automorphism of $C$.
We find that $E(K)=E(\Q)=\Z$, and use this to show that $P$ is defined over a quadratic subfield of $K$.
This leads to the search of $\Q$-rational points on the twists of $C$ over the quadratic subfields of $K$.  

\begin{thm}\label{thm:6}
There are no non-trivial solutions to \eqref{Fermat} over $K$ for $n=6$.
\end{thm}

\begin{proof}
We find it convenient to let 
\[
F_{6} : x^6+y^6=z^6.
\]
That is, $F_{6}$ is the Fermat curve of degree 6. We will prove that $F_{6}(K)=\{(0 : -1 : 1), (-1 : 0 : 1), (0 : 1 : 1), (1 : 0 : 1) \}$, i.e., $F_{6}(K)$ consists only of trivial solutions. 
Suppose $(\alpha:\beta:\gamma)\in F_{6}(K)$ is a non-trivial solution. 
We can assume without loss of generality that $\alpha,\beta,\gamma$ are integral and coprime.
Similar to the proof of Theorem \ref{thm:F9}, observe
\[
\gamma^{12}-(\alpha^6-\beta^6)^2=(\alpha^6+\beta^6)^2-(\alpha^6-\beta^6)^2=4(\alpha\beta)^6.
\]
Let
\[
a=\dfrac{\alpha\beta}{\gamma^2}, \qquad b=\dfrac{\alpha^6-\beta^6}{\gamma^6}.
\]
Then $P=(a,b)\in C(K)$ where
\[
C: y^2=-4x^6+1.
\]
Let 
\[
E: y^2=x^3-4.
\]
This is the elliptic curve with Cremona label \texttt{432b1}. Let 
\[
\pi: C\rightarrow E,\qquad (x,y)\mapsto \left(\dfrac{1}{x^2},\;\dfrac{y}{x^3}\right),\qquad (0,\pm 1)\mapsto 0_{E}.
\]
We checked using \texttt{Magma} that $E$ has rank 1 over $K$ (and $\Q$) and that
\[
E(K)=E(\Q)\cong \Z.
\]
Since $\pi(P)\in E(\Q)$, it follows that $a^2\in\Q$ and hence $b^2\in\Q$. 
If $a=0$ then it's clear that $(\alpha:\beta:\gamma)$ is a trivial solution. 
Observe that $a,b$ are necessarily defined over the same quadratic subfield of $K$ since 
\[
\dfrac{b}{a}\in\Q.
\]
Either $a\in\Q$ and hence $b\in\Q$ or
\[
a=\frac{a^\prime}{\sqrt{d}}, \qquad b=\frac{b^\prime}{\sqrt{d}}, \qquad \text{for }\;d\in\{2,3,6\},\qquad a^\prime,b^\prime\in\Q.
\]
If $a,b\in\Q$ then $P\in C(\Q)$. The Jacobian of $C$ has rank 1 over $\Q$. 
Using the Chabauty implementation in \texttt{Magma}, we find that 
\[
C(\Q)=\{(0,\; \pm 1)\}
\]
and it immediately follows that $(\alpha:\beta:\gamma)$ is a trivial solution.
Thus, $(a^\prime,\; b^\prime d)\in C_{d}(\Q)$ where
\[
C_{d}:\; y^2=-4x^6+d^3,
\]
where $d\in\{2, 3, 6\}$. 
Suppose $d=3$ or $6$. We checked using \texttt{Magma} 
that there are no points on $C_{d}$ defined over $\Q_{2}$. 
Thus $C_{3}(\Q)=C_{6}(\Q)=\emptyset$.  
It remains to determine $C_{2}(\Q)$.
We will work with the model
\begin{equation}\label{eqn:C2}
C_{2}:\; y^2 = -x^6+2.
\end{equation}
We note that the curve $C_{2}$ has genus 2 and the rank of the Jacobian of $C_{2}$ over $\Q$ is 2. Thus, we are unable to determine $C_{2}(\Q)$ using Chabauty.
Instead, we used Bruin's elliptic curve Chabauty method \cite{Bruin03} to do so as we now demonstrate.

Let $\theta=\sqrt[6]{2}$, and note that $\theta$ is a root of the hyperelliptic polynomial
for $C_2$ given in \eqref{eqn:C2}. Let $L=\Q(\theta)$. Consider the map
\[
\varphi : C_2(\Q) \rightarrow L^*/(L^*)^2, \qquad (x,y) \rightarrow (x-\theta) \cdot (L^*)^2.
\]
The method of two-cover descent, due to Bruin and Stoll \cite{BruinStoll}, uses sieving
information to determine a small finite set containing the image of $\varphi$.
This is implemented in \texttt{Magma}, and applying it we find that
\[
\varphi(C_2(\Q)) \subseteq \{ (1+\theta) \cdot (L^*)^2, ~ (1-\theta) \cdot (L^*)^2     \} .
\]
Thus for a rational point $(x,y) \in C_2(\Q)$ we have 
\begin{equation}\label{eqn:xtheta}
x-\theta \; = \; (1\pm \theta) \beta^2
\end{equation}
with $\beta \in L^*$. 
Now let $F=\Q(\sqrt[3]{2})$, and 
note that $x^2- \sqrt[3]{2} = \Norm_{L/F}(x-\theta)$.
Observe that
\[
\Norm_{L/F}(1 \pm \theta)=(1-\theta)(1+\theta)=1-\sqrt[3]{2}.
\]
Taking norms in \eqref{eqn:xtheta} gives
\[
x^2- \sqrt[3]{2} = (1-\sqrt[3]{2}) w^2, \qquad w=\Norm_{L/F}(\beta) \in F^*.
\]
Note the factorisation
\[
C_{2}:\; y^2 = -x^6 + 2 = -(x^2 - \sqrt[3]{2})(x^4+\sqrt[3]{2}x^2+\sqrt[3]{2}^2).
\]
Thus for $(x,y) \in C_2(\Q)$ we have 
\[
x^4+\sqrt[3]{2}x^2+\sqrt[3]{2}^2 = \frac{-y^2}{x^2 - \sqrt[3]{2}}=\frac{-1}{(1-\sqrt[3]{2})} \cdot \frac{y^2}{w^2}.
\]
Let
$\epsilon=-1/(1-\sqrt[3]{2})=1+\sqrt[3]{2}+\sqrt[3]{2}^2 \in F^*$,  and  $z=y/w \in F^*$.
Then, for $(x,y) \in C_2(\Q)$ we have
\begin{equation}	
	\label{eq:6quartic}
		x^4+\sqrt[3]{2}x^2+\sqrt[3]{2}^2=\epsilon z^2.
\end{equation}
Let
\begin{equation}\label{eqn:fromEC}
X = \epsilon x^2\;\text{ and } Y = \epsilon^2 x z.
\end{equation}
Then $(X,Y)\in E_{2}(F)$ where $E_2/F$ is the elliptic curve
\[
E_{2}:\; Y^2 = X^3 +  \epsilon \sqrt[3]{2} X^2 + \epsilon^2 \sqrt[3]{2}^2 X.
\]
Using \texttt{Magma} we found that the Mordell--Weil group is given by
\[
E_2(F) \; = \;  (\Z/2\Z) \cdot (0,0) \oplus \Z \cdot \left(1+ \sqrt[3]{2} + \sqrt[3]{2}^2, 5+ 4 \sqrt[3]{2} + 3\sqrt[3]{2}^2 \right). 
\]
We are interested in points $(X,Y) \in E_2(F)$ which satisfy \eqref{eqn:fromEC} where $(x,y) \in C_2(\Q)$.
In particular, to determine $C_2(\Q)$, it is enough to find all points $Q=(X,Y) \in E_2(F)$
such that $f(Q) \in \Q$, where $f(X,Y)=X/\epsilon$. 
Bruin's elliptic curve Chabauty method \cite{Bruin03} is one that can sometimes be used to 
provably determine all $F$-points $Q$ on an elliptic curve $E$
defined over a number field $F$, such that $f(Q) \in \Q$
for a given non-constant function $f \in F(E)$, provided the
degree $[F:\Q]$ exceeds the rank of $E$ over $F$. 
In our situation, the degree is $[F:\Q]=3$ and the rank of $E$ over $F$
is $1$. We applied the implementation of elliptic curve Chabauty
available in \texttt{Magma} to our $E_2/F$ and $f$. This succeeded
in showing that the only $(X,Y) \in E_2(F)$ with $X/\varepsilon \in \Q$
are
\[
(X,Y)=(0,0), \quad \left(1+ \sqrt[3]{2} + \sqrt[3]{2}^2, 5+ 4 \sqrt[3]{2} + 3\sqrt[3]{2}^2 \right),
\quad
\left(1+ \sqrt[3]{2} + \sqrt[3]{2}^2, -5- 4 \sqrt[3]{2} - 3\sqrt[3]{2}^2 \right).
\]
Thus $X=0$ or $\epsilon$, and hence if $(x,y) \in C_2(\Q)$,
then $x=0$ or $\pm 1$.
%
%
%
It immediately follows that
\[
C_{2}(\Q)=\{(\pm1,\; \pm1)\}.
\]
Thus, $(a^\prime,\; b^\prime)\in \{(\pm 1,\; \pm1)\}$ and if $P=(a,b)\in C(K)$ then $P\in\{(\pm\; 1/{\sqrt{2}},\; \pm\; 1/\sqrt{2})\}$.
Recall that
\[
b=\dfrac{\alpha^6-\beta^6}{\gamma^6},
\]
where $(\alpha:\beta:\gamma)\in F_{6}(K)$. 
It immediately follows that $(b+1)/2$ is a square in $K$. 
For each $b$, we check using \texttt{Magma} that $(b+1)/2$ is not a square in $K$. 
We have reached a contradiction. 
This completes the proof.
\end{proof}

\subsection{$n=4$} Quadratic points on the Fermat quartic have been studied by Aigner~\cite{Ai}, Faddeev~\cite{Fa} and Mordell~\cite{Mo}.
Mordell starts with the knowledge that there are no non-trivial points on the Fermat quartic over $\Q$, and studies points
over all quadratic fields. 
We generalise his method, observing that we can also classify points over quadratic extensions of certain quadratic fields. 
More precisely, if $L$ is any field for which there are no points on
the Fermat quartic, and if the two elliptic curves with Cremona labels \texttt{32a1} and \texttt{64a1} have rank $0$ over $L$, then we give a procedure to write down all the points on the Fermat quartic over quadratic extensions of $L$.\\

In an earlier version of this paper, we conjectured that there are no points on the Fermat quartic over any real biquadratic field. 
We thank Pedro Jos\'e Cazorla Garcia for pointing out to us that the point $(\sqrt{3}, 2, \sqrt{5})$ lies on the Fermat quartic over $\Q(\sqrt{3}, \sqrt{5})$.\\

After the completion of this work, we were made aware that Ishitsuka, Ito and Ohshita \cite[Theorem 7.3]{Ishitsuka2019ExplicitCO} have previously determined all points on the Fermat quartic lying in a quadratic extension of $\Q(\zeta_{8})$.
We thank the authors for making us aware of this. 
Since $\Q(\sqrt{2})$ is contained in $\Q(\zeta_{8})$, this is indeed stronger than the statement of Theorem ~\ref{thm:F4}.
We note that the authors of \cite{Ishitsuka2019ExplicitCO} study the Jacobian of the Fermat quartic over $\Q(\zeta_{8})$ and that the proof of Theorem ~\ref{thm:F4}, extending work of Mordell \cite{Mo}, makes use of a different strategy. 

\begin{thm}
 \label{thm:F4}
The points on the Fermat quartic lying in quadratic extensions of $\Q(\sqrt{2})$
lie in one of $\Q(\sqrt{2},i)$, $\Q(\sqrt{2},\sqrt{-7})$, $\Q(\sqrt[4]{2})$
or $\Q(\sqrt[4]{2}i)$.
\end{thm}

\begin{proof}
Let $L=\Q(\sqrt{2})$, and let $K$ be a quadratic extension of $L$. 
We will determine all points the Fermat quartic $F_{4}: x^4+y^4=1$ in $K$, using the same strategy as Mordell (where, of course, Mordell works with a quadratic extension $K$ of $L=\Q$). Let $t=\frac{1-x^2}{y^2}$, so that $x^2+ty^2=1$. This gives a
parametrisation
\[
x^2=\frac{1-t^2}{1+t^2},\qquad y^2=\frac{2t}{1+t^2}.
\]
We point out that if $x,y\in K$ then $x^2, y^2\in K$ and therefore
so is $t$.\\

Suppose first that $t\in L$. Then $x^2, y^2\in L$. In order for $x$ and $y$ to lie in the same quadratic extension $K$ of $L$, either $x\in L$, $y\in L$ or $x/y\in L$. This means that one of
\[
\frac{1-t^2}{1+t^2},\quad \frac{2t}{1+t^2}\quad \text{ or }\quad \frac{2t}{1-t^2}
\]
is a square in $L$. Equivalently, $(1-t^2)(1+t^2)$, $2t(1+t^2)$ or $2t(1-t^2)$ is a square in $L$. These correspond to $L$-rational points of one of the curves
\[
    u^2=(1-t^2)(1+t^2),\quad u^2=2t(1+t^2), \quad u^2=2t(1-t^2).
\]
Both of the first two possibilities are isomorphic to $E_{1}: y^2=x^3+4x$ 
(the elliptic curve with Cremona label \texttt{32a1}) via the maps 
$(t,u)\mapsto\left(\frac{2t+2}{1-t},\frac{u}{(1-t)^2}\right)$ and
$(t,u)\mapsto(2t,2u)$ respectively, and the third to $E_{2}: y^2=x^3-4x$
(the elliptic curve with Cremona label \texttt{64a1}) via $(t,u)\mapsto(-2t,2u)$. We checked, using \texttt{Magma}, that $E_{1}$ and $E_{2}$ have rank 0 over $L$. We first consider $E_{1}$ and find
\[
    E_{1}(L)=E_{1}(\Q)=\{\mathcal{O}, (0, 0), (2, \pm 4)\}.
\]
We find that these points correspond on the first curve to $t=\pm1$ and $t=0$,
and on the second to $t=0$, $t=1$ and $t=\infty$. These values of $t$ correspond to 
\[
(x^2,y^2)= \{(1,0), (-1,0), (0,1), (0,-1)\},
\]
corresponding to points on $F_{4}$ defined over $\Q$ or $\Q(i)$.
Similarly,
\begin{align*}
    E_{2}(L) = \{ \mathcal{O}, (0, 0), (\pm 2,0)\}&\\
    &\cup \{(2+2\sqrt{2},\pm(4+4\sqrt{2}), (2-2\sqrt{2},\pm(4-4\sqrt{2})    \},
\end{align*}
and the rational points correspond to $t=\pm1$ and $t=0$, and the point at infinity to $t=\infty$, as before. 
The points in $E(L)\setminus E(\Q)$ correspond to $t=-1\pm\sqrt{2}$, and
these give 
\[
(x^2,y^2)\in \{(1/\sqrt{2},1/\sqrt{2}),\; (-1/\sqrt{2},-1/\sqrt{2})\}, 
\]
corresponding to points on $F_{4}$ defined over $\Q(\sqrt[4]{2})$ or $\Q(\sqrt[4]{2}i)$.\\

We now suppose $t\in K,\; t\not\in L$.
We write $F(t)=t^2+\beta t+\gamma$ for the minimal polynomial of $t$
over $L$, so $\beta,\gamma\in L$.
We let $A=(1+t^2)xy$ and $B=(1+t^2)y$, so that
\[
A^2=2t(1-t^2), \qquad B^2=2t(1+t^2).
\]
Since $A^2, B^2\in K$ and $K=L(t)$, we can write
\[
A=\lambda+\mu t, \quad B=\lambda^\prime +\mu^{\prime}t,
\qquad \lambda,\; \mu,\; \lambda^\prime,\; \mu' \in L.
\]
Comparing the two expressions for $A$ yields
\[
(\lambda+\mu t)^2=2t(1-t^2).
\]
In particular, the equation
$$(\lambda+\mu z)^2-2z(1-z^2)=0$$
has a root $z=t$. 
As the equation is defined over $L$, the left-hand side is divisible
by the minimal polynomial $F(z)$, and, as this is a cubic, we have
\begin{equation}\tag{M1}
(\lambda+\mu z)^2-2z(1-z^2)=F(z)(\rho+\sigma z),
\end{equation}
a factorisation over $L$ (so $\rho, \sigma\in L$).
Then $z=-\rho/\sigma$ is a solution to the left-hand side of (M1) defined over $L$. In particular,
we have a solution with $z\in L$ to
$$Y^2=2z(1-z^2)=-2z^3+2z,$$
where $Y=\lambda+\mu z\in L$. Thus we get an $L$-point on the elliptic curve $Y^2=-2X^3+2X$, which is isomorphic to the elliptic curve $E_{2}$, and the points in $E_{2}(L)$ correspond to $z=\pm1$, $z=0$ and $z=-1\pm\sqrt{2}$.
In exactly the same way, looking at $B^2$, we will get a solution over $L$ to
\begin{equation}\tag{M2}
(\lambda'+\mu' z)^2-2z(1+z^2)=F(z)(\rho'+\sigma'z),
\end{equation}
and therefore a solution over $L$
to $Y^2=2z(1+z^2)$, which is isomorphic to $E_{1}$. The points in $E_{2}(L)$ correspond to $z=0$ and $z=1$.

We will now consider all these cases, as in Mordell. We write $(z_1,z_2)$ for the situation where the equation (M1) is solved by
$z_1$ and equation (M2) is solved by $z_2$.
We remark that these calculations are quite involved, and we therefore omit some details.

\subsubsection*{\textbf{Case 1.} $(-1, 1)$} This is Mordell's case (VI). If $z_{1}=-1$ is a root of the left-hand side of (M1) then $\lambda+\mu=0$ and, since $-1$ must then be a root of the right-hand side of (M1), it follows that $\rho+\sigma=0$. 
Similarly, if $z_{1}=1$ is a root of the left-hand side of (M2) then $\lambda'+\mu'=2$, $\rho'+\sigma'=0$.
Equation (M1) is
$$\lambda^2(1+z)-2z(1-z)=\rho F(z)$$
(after dividing out by $1-z$). 
We can rewrite the left-hand side of (M2) as $(2-\mu'+\mu'z)^2-2z-2z^3=\rho'(1-z)F(z)$. Thus, after dividing through by $1-z$, we get
$$(\text{M2}): 2(z^2+z+2)-4\mu'+\mu'^2(1-z)=\rho'F(z).$$
Both (M1) and (M2) have the same coefficient of $z^2$, so $\rho=\rho'$. Comparing constant
terms and $z$ terms:
$$\lambda^2=(2-\mu')^2,\quad \lambda^2-2=2-\mu'^2,$$
so either $(\lambda, \mu')=(0,2)$ or $(\lambda, \mu')=(\pm 2,0)$. In the first case, (M1) becomes $-2z(1-z)=\rho\cdot F(z)$ but this contradicts the irreducibility of $F(z)$. In the second case,
\[
(\text{M1}): \rho F(z)=4(1+z)-2z(1-z)=2(z^2+z+2),
\]
so $F(z)=z^2+z+2$. Thus, $t=\frac{-1\pm\sqrt{-7}}{2}$ and $K=L(\sqrt{-7})$.

\subsubsection*{\textbf{Case 2.} $(-1, 0)$} 
This is Mordell's case (III).
In order for $z=-1$ to be a root of the
left-hand side of (M1), we need $(\lambda-\mu)^2=0$.
So $\lambda-\mu=0$.
Similarly, for $z=0$ to be a root of the left-hand side of
(M2), we need $\lambda'=0$.
Then for the left-hand side of (M1) to have $-1$
as a root, the same will be true of the right-hand side, 
so $\rho-\sigma=0$.
Equation (M1) is then divisible by $(1+z)$, and dividing through, we get
\[
(\text{M1}): \lambda^2(1+z)-2z(1-z)=\rho\cdot F(z).
\]
We rewrite this as
\[
(\text{M1}): 2z^2+(\lambda^2-2)z+\lambda^2=\rho\cdot F(z).
\]
In order for $z=0$ to be a root of the left-hand side of (M2), it must be that $\lambda^\prime=0$, and thus
\[
(\text{M2}): -2z^2+\mu'{}^2z-2=\sigma'F(z).
\]
The right-hand sides of (M1) and (M2) differ by a constant, and upon comparing the $z^2$ coefficients on the left-hand sides, we
see that they differ by a factor of $-1$. 
Then comparing the constant term, we 
get $\lambda^2=2$. Thus $\lambda=\mu=\pm\sqrt{2}$. 
The coefficient of $z$ in the first equation is $\lambda^2-2$, and 
the coefficient of $z$ in the second is $\mu'{}^2$, so $\mu'=0$. Then
$Y=\lambda'+\mu't=0$. But $Y^2=2t(1+t^2)$, so this means that $t=0$, 
contradicting $t\notin L$, or $(1+t^2)$ in which case $t=i$ and $K=L(i)$.\\

For the remaining pairs $(z_{1}, z_{2})$, in each case, after performing a similar analysis, we reach a contradiction to the fact $\lambda, \mu, \lambda^\prime, \mu^\prime \in L$, and thus no solutions are found in these cases.
\end{proof}

This completes the proof of Theorem \ref{mainthm}.

\section{More general real biquadratic fields}\label{sec:genfields}

We give examples of obstacles that arise in generalising the 
proof of Theorem \ref{mainthm} to more general real biquadratic fields. 
As in the proof of Theorem \ref{mainthm}, we apply level-lowering (Theorem \ref{levellowering})
to the Frey curve \eqref{eqn:Frey} for $p\geq 17$ and $E_{13,\epsilon}$ for $p=13$.

\subsection{$K=\Q(\sqrt{2},\sqrt{5})$} In order to apply level-lowering (Theorem 
\ref{levellowering}), one needs to demonstrate the modularity 
of the Frey curve over $K$. 
It has not yet been proven that elliptic curves 
over totally real quartic fields containing $\sqrt{5}$ are modular 
(see \cite[Section 7.1]{box2021elliptic} for a discussion concerning 
this problem). 
We remark however that establishing the modularity of the 
Frey curve over this particular field $K$ may be possible through the use 
of \cite[Theorem 7]{freitas2014elliptic}.

\subsection{$K=\Q(\sqrt{2},\sqrt{7})$} 
Write $\OO_K$ for the ring of integers of $K$. 
A straightforward computation in \texttt{Magma} 
returns that $K$ has class number $1$, and $2\OO_{K}=\fP^{4}$. 
A straightforward generalisation of Lemmata \ref{lem:poss}, \ref{lem:13surj} and \ref{lem:13} 
returns that the lowered level is $\fP^t$ where 
$t=1, 5, 8$ or $16$. 
In particular the dimension of Hilbert newforms 
of parallel weight $2$ and level $\fP^{16}$ is $40960$ 
making the elimination step currently computationally infeasible 
in this case.

\subsection{$K=\Q(\sqrt{2},\sqrt{11})$}

Write $\OO_K$ for the ring of integers of $K$. 
A straightforward computation in \texttt{Magma} 
returns that $K$ has class number $1$, and $2\OO_{K}=\fP^{4}$. 
By a direct generalisation of the techniques outlined in Section \ref{sec:irred} 
it is straightforward to see that $\modpg$ 
is irreducible for $p\geq 13$.

A straightforward generalisation of Lemmata \ref{lem:poss}, \ref{lem:13surj} and \ref{lem:13}  
returns that the lowered level is $\fP^t$ where 
$t=1, 4$ or $5$. 
As is true for $\Q(\sqrt{2}, \sqrt{3})$, there are no Hilbert newforms of parallel weight $2$ and level $\fP$ over $K$.
There are $44$ Hilbert newforms of parallel weight $2$ and 
level $\fP^4$ and $76$ Hilbert newforms of parallel weight $2$ and level $\fP^5$ over $K$. 
In order to get a contradiction, we make use of the standard method of eliminating newforms given by the following 
Lemma \cite[Lemma 7.1]{FLTsmall}.

\begin{lem}[Freitas and Siksek]\label{lem:newformelim}
    Let $K$ be a totally real field, and let $p \ge 5$ be a prime. 
    Let $E$ be an elliptic curve over $K$
    of conductor $\cN$,
    and let $\ff$ be a newform of parallel weight $2$
    and level $\cN_p$.
    Let $t$ be a positive integer satisfying $t \mid \# E(K)_{\mathrm{tors}}$.
    Let $\fq\nmid t \mathcal{N}_{p}$ be a prime ideal of $\OO_K$ and let 
    \[
    \mathcal{A_{\mathfrak{q}}}=\{a\in\Z:\quad \lvert a \rvert\leq 2\sqrt{\Norm(\mathfrak{q})}\;, 
    \quad \Norm(\mathfrak{q})+1-a\equiv 0 \pmod{t} \}.
    \]
    If $\modpg\sim \overline{\rho}_{\mathfrak{f},\varpi}$ 
    then $\varpi$ divides the principal ideal 
    \[
    B_{\mathfrak{f},\mathfrak{q}}=\Norm(\mathfrak{q})((
    \Norm(\frak{q})+1)^2-a_{\mathfrak{q}}(\mathfrak{f})^2)\prod\limits_{a\in\mathcal{A_{\mathfrak{q}}}}(a-a_{\mathfrak{q}}(\mathfrak{f}))\cdot \mathcal{O}_{\Q_{\mathfrak{f}}}.
    \]
    \end{lem}
We briefly explain how to apply Lemma \ref{lem:newformelim}.
Namely let  
    \[
    B_{\mathfrak{f}}=\sum\limits_{\mathfrak{q}\in T}B_{\mathfrak{f}, \mathfrak{q}},
    \]
    where $T$ is a 
    small set of primes $\mathfrak{q}\nmid t \mathcal{N}_{p}$. 
Let $C_{\mathfrak{f}}=\Norm_{\Q_{\mathfrak{f}}/\Q}(B_{\mathfrak{f}})$. 
Then Lemma \ref{lem:newformelim} asserts that $p\mid C_{\mathfrak{f}}$.
We wrote a short program to implement Lemma \ref{lem:newformelim} in 
\texttt{Magma} with $\mathcal{N}_p=\fP^4$ or $\fP^5$, with $t=4$ and $T$ equal to the set of prime ideals $\mathfrak{q}\neq \fP$ of $K$ with 
norm less than 90.
From this implementation we found that if $\Gisom$, where $E$ is our Frey curve and $\mathfrak{f}$ 
is a newform of level $\mathcal{N}_p$ then $p=2$ or $3$.

We remark that the proofs of Theorems \ref{thm:F9} and \ref{thm:6} do not readily generalise to $K$.
In combination with the remarks made in 
Section \ref{modapproach}, this leads to the following result.

\begin{thm}
    Let $K=\Q(\sqrt{2},\sqrt{11})$. 
    There are no non-trivial solutions to \eqref{Fermat} over $K$ for all primes $n\geq 5$.
    \end{thm}

\bibliographystyle{abbrv}
\bibliography{Fermat}
\nocite{MR1484478}
\end{document}